\documentclass[a4paper,11pt,reqno]{article}
\usepackage{}
\usepackage{amsmath}
\usepackage{amssymb}
\usepackage{amsthm}
\usepackage{amstext}
\usepackage{delarray}

\newtheorem{proposition}{Proposition}[section]
\newtheorem{theorem}[proposition]{Theorem}
\newtheorem{lemma}[proposition]{Lemma}

\begin{document}
\date{}
\title{\bf Perturbation analysis of $A_{T,S}^{(2)}$ on Hilbert spaces}
\author{Fapeng Du\thanks{E-mail: jsdfp@163.com}\\
   School of Mathematical \& Physical Sciences, Xuzhou Institute of Technology\\
  Xuzhou 221008, Jiangsu Province, P.R. China\\
 \and
 Yifeng Xue\thanks{Corresponding author, E-mail: yfxue@math.ecnu.edu.cn} \\
 Department of mathematics, East China Normal University\\
  Shanghai 200241, P.R. China
}
\maketitle

\begin{abstract}
In this paper, we investigate the perturbation analysis  of $A_{T,S}^{(2)}$ when $T,\,S$ and $A$
have some small perturbations on Hilbert spaces. We present the conditions that make the perturbation of $A_{T,S}^{(2)}$ is stable.
The explicit representation for the perturbation of $A_{T,S}^{(2)}$ and
the perturbation bounds are also obtained.

\hspace*{1mm}


\end{abstract}

\section{Introduction}

Let $X,Y$ be Banach spaces and let $B(X,Y)$ denotes the set of bounded linear operators from $X$ to $Y$. For an operator
$A \in B(X,Y)$, let $R(A)$ and $N(A)$ denote the range and kernel of $A$, respectively. Let $T$ be a closed subspace of
$X$ and $S$ be a closed subspace of $Y$. Recall that $A_{T,S}^{(2)}$ is the unique operator $G$ satisfying
\begin{equation}\label{1eqa}
GAG=G,\quad R(G)=T,\quad N(G)=S.
\end{equation}
It is known that (\ref{1eqa}) is equivalent to the following condition:
\begin{equation}\label{1eqb}
N(A)\cap T=\{0\},\quad AT\dotplus S=Y
\end{equation}
(cf. \cite{DS, DSW}). It is well--known that the commonly five kinds of generalized inverse:
the Moore--Penrose inverse $A^+$, the weighted Moore--Penrose inverse $A^+_{MN}$,
the Drazin inverse $A^D$, the group inverse $A^\#$ and the Bott--Duffin inverse $A^{(-1)}_{(L)}$
can be reduced to a $A_{T,S}^{(2)}$ for certain choices of $T$ and $S$.

The perturbation analysis of $A_{T,S}^{(2)}$ have been studied by
several authors (see \cite{WW1, WW2}, \cite{ZW1, ZW2}) when $X$ and
$Y$ are of finite--dimensional. A lot of results about the error
bounds have been obtained. When $X$ and $Y$ are of infinite--dimensional Banach spaces, the perturbation analysis
of $A_{T,S}^{(2)}$ for small perturbation of $T$, $S$ and $A$ has been done in \cite{DX}.

In this paper, we assume that $X$ and $Y$ are all Hilbert spaces over the complex field $\mathbb C$. Using the theory of
stable perturbation of generalized inverses established by G. Chen and Y. Xue in \cite{CX, CWX}, we
will give the upper bounds of $\|\bar A_{T^\prime,S^\prime}^{(2)}\|$ and $\|\bar A_{T^\prime,S^\prime}^{(2)}-A_{T,S}^{(2)}\|$
respectively for certain $T'$, $S'$ and $\bar A$. The results in this paper improve \cite[Theorem 4.4.7]{XUE}.

\section{Preliminaries}
\setcounter{equation}{0}

Let $H$ be a complex Hilbert space. Let $V$ be a closed subspace of $H$. We denote by $P_V$ the orthogonal projection of
$H$ onto $V$. Let $M,\,N$ be two closed subspaces in $H$. Set
$$
\delta(M,N)=\begin{cases}\sup\{dist(x,N)\,\vert\,x\in M,\,\|x\|=1\},\quad &M\not=\{0\}\\ 0 \quad& M=\{0\}
\end{cases},
$$
where $dist(x,N)=\inf\{\|x-y\|\,\vert\,y\in N\}$. The gap $\hat\delta(M,N)$ of $M,\,N$ is given by
$\hat{\delta}(M,N)=\max\{\delta(M,N),\delta(N,M)\}$. For convenience, we list some properties about $\delta(M,N)$
and $\hat\delta(M,N)$ which come from \cite{TK} as follows.

\begin{proposition}[\text{\cite{TK}}]\label{2P1}
Let $M,\,N$ be closed subspaces in a Hilbert space $H$.
\begin{enumerate}
  \item[$(1)$] $\delta(M,N)=0$ if and only if $M\subset N$
  \item[$(2)$] $\hat{\delta}(M,N)=0$ if and only if $M=N$
  \item[$(3)$] $\hat{\delta}(M,N)=\hat{\delta}(N,M)$
  \item[$(4)$] $0\leq \delta(M,N)\leq 1$, $0\leq \hat{\delta}(M,N)\leq 1$
   \item[$(5)$] $\hat{\delta}(M,N)=\|P_M-Q_N\|$.
\end{enumerate}
\end{proposition}

Let $A\in B(X,Y)$. If there is $C\in B(Y,X)$ such that $ACA=A$ and $CAC=C$, we call $C$ is a generalized inverse of $A$
and is denoted by $A_{GI}^+$. In this case, $R(A)$ is closed in $Y$.

Recall that $A$ is Moore--Penrose invertible, if there is $B\in B(Y,X)$ such that
\begin{equation}\label{2eqa}
ABA=A,\ BAB=B,\ (AB)^*=AB,\ (BA)^*=BA.
\end{equation}
The operator $B$ in (\ref{2eqa}) is called the Moore--Penrose inverse of $A$ and is denoted as $A^+$.
It is well--known that $A$ is Moore--Penrose invertible iff $R(A)$ is closed in $Y$. Thus, $A$ is Moore--Penrose
invertible iff $A_{GI}^+$ exists.

Let $A,\,\delta A\in B(X,Y)$ and put $\bar A=A+\delta A$. Recall that $\bar A$ is the stable perturbation of
$A$ if $R(\bar A)\cap R(A)^\perp=\{0\}$.

The next lemma illustrates some equivalent conditions of the stable perturbation.
\begin{lemma}[\cite{XCC, DX1}]\label{2L1}
Let $A\in B(X,Y)$ with $R(A)$ closed and $\delta A \in B(X,Y)$ with $\|A^+\|\|\delta A\|<1$. Put $\bar{T}=T+\delta T$.

$(A)$ The following conditions are equivalent.
\begin{enumerate}
\item[$(1)$] $R(\bar{A})\cap R(A)^\perp=\{0\}$
\item[$(2)$] $N(\bar{A})^\perp \cap N(A)=\{0\}$
\item[$(3)$] $R(\bar{A})$ is closed and $\bar{A}^+_{GI}=A^+(I+\delta AA^+)^{-1}=(I+A^+\delta A)^{-1}A^+$
\end{enumerate}

$(B)$ If $\bar A$ is the stable perturbation of $A$, then $R(\bar A)$ is closed and
$$
\|\bar A^+\|\leq\frac{\|A^+\|}{1-\|A^+\|\|\delta A\|},\
\|\bar A^+-A^+\|\leq \frac{1+\sqrt{5}}{2}\|\bar{A}^+\|\|A^+\|\|\delta A\|.
$$
\end{lemma}

\begin{lemma}\label{2L2}
Let $A \in B(X,Y)$ with $R(A)$ closed. If $Z\in B(Y,X)$ satisfies the conditions:
$AZA=A$ and $ZAZ=Z$, then $A^+=P_{N(A)^\perp}ZP_{R(A)}$.
\end{lemma}
\begin{proof} We can check that $P_{N(A)^\perp}ZP_{R(A)}$ satisfies the definition of the Moore--Penrose inverse of $A$.
\qed
\end{proof}

The following result is known when $X,\,Y$ are all of finite--dimensional (cf. \cite{BG}).
\begin{lemma}\label{2L3}
Let $A \in B(X,Y)$ and $T\subset X ,S\subset Y$ be closed subspaces. If $A_{T,S}^{(2)}$ exists, then
$A_{T,S}^{(2)}=(P_{S^\perp}AP_T)^+$ with $R(A_{T,S}^{(2)})=T$ and $N(A_{T,S}^{(2)})=S$.
\end{lemma}
\begin{proof} The existence of $A_{T,S}^{(2)}$ implies that $N(A)\cap T=\{0\}$, $AT$ is closed and $Y=AT\dotplus S$.
Let $P\colon Y\rightarrow S$ be the idempotent operator. Since $R(P)=S$ and $R(I_Y-P)=AT$, it
follows that $PP_S=P_S$, $P_SP=P$ and $(I_Y-P)AT=AT$. Noting that
\begin{align*}
(I_Y-P)(I_Y-P_S)&=I_Y+PP_S-P_S-P=I_Y-P\\
(I_Y-P_S)(I_Y-P)&=I_Y-P-P_S+P_SP=I_Y-P_S,
\end{align*}
we have
$$
R(I_Y-P_S)=(I_Y-P_S)(R(I_Y-P))=(I_Y-P_S)AT=P_{S^\perp}AT
$$
and hence $R(P_{S^{\bot}}AP_T)=R(P_{S^{\bot}})=S^\perp$ is closed.

Let $x\in T$ and $P_{S^{\bot}}Ax=0$. Then $(I_Y-P)Ax=Ax,\ Ax=P_SAx$ and hence
$0=PAx=PP_SAx=P_SAx=Ax$. Since $N(A)\cap T=\{0\}$, we have $x=0$ and consequently,
$N(P_{S^{\bot}}AP_T)=T^{\bot}$. Therefore, $(P_{S^{\bot}}AP_T)^+$ exists and
\begin{align}
\label{2eqb}R((P_{S^{\bot}}AP_T)^+)&=(N(P_{S^{\bot}}AP_T))^{\bot}=T\\
\label{2eqc}N((P_{S^{\bot}}AP_T)^+)&=(R(P_{S^{\bot}}AP_T))^{\bot}=S.
\end{align}
Since
$$
(P_{S^\perp}AP_T)^+ P_{S^{\bot}}=(P_{S^{\bot}}AP_T)^+=P_T(P_{S^{\bot}}AP_T)^+,
$$
by (\ref{2eqb}) and (\ref{2eqc}), it follows that
\begin{align*}
(P_{S^{\bot}}AP_T)^+&=(P_{S^{\bot}}AP_T)^+(P_{S^{\bot}}AP_T)(P_{S^{\bot}}AP_T)^+\\
&=(P_{S^{\bot}}AP_T)^+ A(P_{S^{\bot}}AP_T)^+
\end{align*}
and so that $A_{T,S}^{(2)}=(P_{S^{\bot}}AP_T)^+$.\qed
\end{proof}

\begin{lemma}[{\cite[Theorem 11,P100]{VM}}]\label{2L4}
Let $M$ be a complemented subspace of $H$. Let $P\in B(H)$ be an idempotent operator
with $R(P)=M$. Let $M^\prime$ be a closed subspace of $H$ satisfying
$\hat{\delta}(M,M^\prime)<\dfrac{1}{1+\|P\|}$. Then $M^\prime$ is complemented, that is,
$H=R(I-P)\dotplus M'$.
\end{lemma}

\section{main result}
\setcounter{equation}{0}

We begin with the key lemma as follows.
\begin{lemma}\label{P1}
Let $A \in B(X,Y)$. Let $T\subset X$ and $S\subset Y$ be closed subspaces such that $A_{T,S}^{(2)}$ exists.
Let $T'$ be a closed subspace of $X$ such that $\hat{\delta}(T,T^\prime)< \dfrac{1}{1+\|A\|\|A_{T,S}^{(2)}\|}$.
Then
$$
\hat{\delta}(AT,AT^\prime)\leq \frac{\|A\|\|A_{T,S}^{(2)}\|\hat{\delta}(T,T^\prime)}{1-(1+\|A\|\|A_{T,S}^{(2)}\|)
\hat{\delta}(T,T^\prime)}.
$$
\end{lemma}

\begin{proof} First we show $\delta(AT,AT^\prime)\leq \|A\|\|A_{T,S}^{(2)}\|\delta(T,T^\prime)\leq
\|A\|\|A_{T,S}^{(2)}\|\hat\delta(T,T^\prime)$.

Let $x\in T$. Then $x=A_{T,S}^{(2)}Ax$ and $\|x\|\leq\|A_{T,S}^{(2)}\|\|Ax\|$. For any $y\in T^\prime$, we have
$\|Ax-Ay\|\leq \|A\|\|x-y\|$. So
\begin{align*}
dist(Ax,AT^\prime)&=\inf_{y\in T^\prime}\|Ax-Ay\|\leq \|A\|\inf_{y\in T^\prime}\|x-y\|\\
&=\|A\|dist(x,T^\prime)\leq\|A\|\|x\|\delta(T,T')\\
&\leq\|A\|\|A_{T,S}^{(2)}\|\|Ax\|\delta(T,T').
\end{align*}
This means that $\delta(AT,AT^\prime)\leq \|A\|\|A_{T,S}^{(2)}\|\delta(T,T^\prime)\leq
\|A\|\|A_{T,S}^{(2)}\|\hat\delta(T,T^\prime)$.

Next we show
$$
\delta(AT^\prime,AT)\leq \frac{\|A\|\|A_{T,S}^{(2)}\|\hat\delta(T,T')}{1-(1+\|A\|\|A_{T,S}^{(2)}\|)\hat\delta(T,T')}
$$
when $\hat{\delta}(T,T^\prime)< \dfrac{1}{1+\|A\|\|A_{T,S}^{(2)}\|}.$

For $x^\prime \in T^\prime$ and $x\in T$, we have
\begin{align*}
\|Ax^\prime\|&=\|A(x^\prime-x+x)\|\geq\|Ax\|-\|A\|\|x^\prime-x\| \\
&\geq \|A_{T,S}^{(2)}\|^{-1}\|x\|-\|A\|\|x^\prime-x\| \\
&\geq \|A_{T,S}^{(2)}\|^{-1}\|x^\prime\|-\|A_{T,S}^{(2)}\|^{-1}\|x^\prime-x\|-\|A\|\|x^\prime-x\|\\
&\geq \|A_{T,S}^{(2)}\|^{-1}\|x^\prime\|-(\|A_{T,S}^{(2)}\|^{-1}+\|A\|)\|x^\prime-x\|,
\end{align*}
Thus,
$$
(\|A_{T,S}^{(2)}\|^{-1}+\|A\|)\|x^\prime-x\|\geq \|A_{T,S}^{(2)}\|^{-1}\|x^\prime\|-\|Ax^\prime\|
$$
and consequently,
$$
\|A_{T,S}^{(2)}\|^{-1}\|x^\prime\|-\|Ax^\prime\|\leq\|x^\prime\|(\|A_{T,S}^{(2)}\|^{-1}+\|A\|)\delta(T',T),
$$
that is,
\begin{equation}\label{3eqa}
\|A_{T,S}^{(2)}\|\|Ax^\prime\|\geq \big[1-(1+\|A\|\|A_{T,S}^{(2)}\|)\delta(T',T)\big]\|x'\|.
\end{equation}
Therefore,
\begin{align*}
dist(Ax^\prime, AT)&\leq \|A\|dist(x^\prime,T)\leq\|A\|\|x'\|\delta(T',T)\\
&\leq\frac{\|A\|\|Ax'\|\|A_{T,S}^{(2)}\|\hat\delta(T,T')}{1-(1+\|A\|\|A_{T,S}^{(2)}\|)\hat\delta(T,T')},
\end{align*}
i.e., $\delta(AT',AT)\leq\dfrac{\|A\|\|A_{T,S}^{(2)}\|\hat\delta(T,T')}{1-(1+\|A\|\|A_{T,S}^{(2)}\|)\hat\delta(T,T')}$
when $\hat{\delta}(T,T^\prime)<\dfrac{1}{1+\|A\|\|A_{T,S}^{(2)}\|}.$

The final assertion follows from above arguments.
\qed\end{proof}

\begin{proposition}\label{3P1}
Let $A \in B(X,Y)$ and $T\subset X$, $S\subset Y$ be closed subspaces such that $A_{T,S}^{(2)}$ exists.
Let $T'$ be a closed subspace of $X$ such that $\hat{\delta}(T,T^\prime)< \dfrac{1}{(1+\|A\|\|A_{T,S}^{(2)}\|)^2}$.
Then $A_{T^\prime,S}^{(2)}$ exists and
\begin{enumerate}
\item[$(1)$] $A_{T^\prime,S}^{(2)}=P_{T^\prime}(I_X+A_{T,S}^{(2)}P_{S^\perp}A(P_{T^\prime}-P_{T}))^{-1}A_{T,S}^{(2)}
P_{S^\perp}$.
\item[$(2)$] $\|A_{T^\prime,S}^{(2)}\|\leq \dfrac{\|A_{T,S}^{(2)}\|}{1-\|A_{T,S}^{(2)}\|\|A\|\hat{\delta}(T,T^\prime)}$.
\item[$(3)$] $\|A_{T^\prime,S}^{(2)}-A_{T,S}^{(2)}\|\leq \dfrac{1+\sqrt{5}}{2}\|A_{T^\prime,S}^{(2)}\|\|A_{T,S}^{(2)}\|
\|A\|\hat{\delta}(T,T^\prime).$
\end{enumerate}
\end{proposition}
\begin{proof} By (\ref{3eqa}), $N(A)\cap T'=\{0\}$ when $\hat\delta(T,T^\prime)<\dfrac{1}{(1+\|A\|\|A_{T,S}^{(2)}\|)^2}$.

Let $P=AA_{T,S}^{(2)}$. Then $P$ is idempotent from $Y$ onto $AT$ along $S$. By Lemma \ref{P1}, we have
$$
\hat{\delta}(AT,AT^\prime)\leq \frac{\|A\|\|A_{T,S}^{(2)}\|\hat{\delta}(T,T^\prime)}{1-(1+\|A\|\|A_{T,S}^{(2)}\|)\hat{\delta}(T,T^\prime)}
<\frac{1}{1+\|A\|\|A_{T,S}^{(2)}\|}\leq \frac{1}{1+\|P\|}
$$
when $\hat{\delta}(T,T^\prime)< \dfrac{1}{(1+\|A\|\|A_{T,S}^{(2)}\|)^2}.$
So $AT^\prime$ is complemented and $AT^\prime \dotplus S=Y$ by Lemma \ref{2L4}. Therefore. $A_{T',S}^{(2)}$ exists and
$A_{T',S}^{(2)}=(P_{S^\perp}AP_{T'})^+$ by Lemma \ref{2L3}.

Set $B=P_{S^\perp}AP_{T}$, $\bar{B}=B+P_{S^\perp}A(P_{T^\prime}-P_{T})=P_{S^\perp}AP_{T'}$. Then $N(B^+)=S$
and $R(\bar B)=((N(\bar B^+))^\perp=S^\perp$. So $R(\bar{B})\cap N(B^+)=\{0\}$, that is, $\bar{B}$ is the stable perturbation of $B$.

From Proposition \ref{2P1} (5), we have
$$
\|B^+P_{S^\perp}A(P_{T^\prime}-P_{T})\| \leq \|A_{T,S}^{(2)}\|\|A\|\|P_{T^\prime}-P_{T}\|
=\|A_{T,S}^{(2)}\|\|A\|\hat{\delta}(T,T^\prime)<1.
$$
Hence, by Lemma \ref{2L1} and Lemma \ref{2L2}, we have
\begin{align*}
A_{T^\prime,S}^{(2)}=\bar{B}^+&=P_{N(\bar{B})^\perp}(I+B^+P_{S^\perp}A(P_{T^\prime}-P_{T}))^{-1}B^+ P_{R(\bar{B})}\\
&=P_{T^\prime}(I+A_{T,S}^{(2)}P_{S^\perp}A(P_{T^\prime}-P_{T}))^{-1}A_{T,S}^{(2)}P_{S^\perp},
\end{align*}
$\|A_{T^\prime,S}^{(2)}\|\leq \dfrac{\|A_{T,S}^{(2)}\|}{1-\|A_{T,S}^{(2)}\|\|A\|\hat{\delta}(T,T^\prime)}$ and
\begin{align*}
\|A_{T^\prime,S}^{(2)}-A_{T,S}^{(2)}\|&=\|\bar{B}^+ - B^+ \| \\
&\leq \frac{1+\sqrt{5}}{2}\|A_{T^\prime,S}^{(2)}\|\|A_{T,S}^{(2)}\|\|P_{S^\perp}A(P_{T^\prime}-P_{T})\|\\
&\leq \frac{1+\sqrt{5}}{2}\|A_{T^\prime,S}^{(2)}\|\|A_{T,S}^{(2)}\|\|A\|\|P_{T^\prime}-P_{T}\|\\
&=\frac{1+\sqrt{5}}{2}\|A_{T^\prime,S}^{(2)}\|\|A_{T,S}^{(2)}\|\|A\|\hat{\delta}(T,T^\prime).
\end{align*}\qed
\end{proof}

Similar to Proposition \ref{3P1}, we have
\begin{proposition}\label{3P2}
Let $A \in B(X,Y)$ and let $T\subset X$, $S\subset Y$ be closed subspaces such that $A_{T,S}^{(2)}$ exists.
Let $S'\subset Y$ be a closed subspace such that $\hat{\delta}(S,S^\prime)< \dfrac{1}{2+\|A\|\|A_{T,S}^{(2)}\|}$.
Then $A_{T,S^\prime}^{(2)}$ exists and
\begin{enumerate}
\item[$(1)$] $A_{T,S^\prime}^{(2)}=P_T(I_X+A_{T,S}^{(2)}(P_{(S^\prime) ^\perp}-P_{S^\perp})AP_{T})^{-1}A_{T,S}^{(2)}P_{(S^\prime)^\perp}$.
\item[$(2)$] $\|A_{T,S^\prime}^{(2)}\|\leq \dfrac{\|A_{T,S}^{(2)}\|}{1-\|A_{T,S}^{(2)}\|\|A\|\hat{\delta}(S,S^\prime)}$.
\item[$(3)$] $\|A_{T,S^\prime}^{(2)}-A_{T,S}^{(2)}\|\leq \dfrac{1+\sqrt{5}}{2}\|A_{T,S^\prime}^{(2)}\|\|A_{T,S}^{(2)}\|\|A\|
\hat{\delta}(S,S^\prime)$.
\end{enumerate}
\end{proposition}
\begin{proof}
Note that $Q=I_Y-AA_{T,S}^{(2)}$ is an idempotent operator from $Y$ onto $S$ along $AT$ and
$$
\hat{\delta}(S,S^\prime)< \frac{1}{2+\|A\|\|A_{T,S}^{(2)}\|}\leq \frac{1}{1+\|I_Y-Q\|}.
$$
So $Y=AT\dotplus S'$ by Lemma \ref{2L4} and hence $A_{T,S'}^{(2)}$ exists with $A_{T,S'}^{(2)}=(P_{{S'}^\perp}AP_T)^+$.
Using similar methods in the proof of Proposition \ref{3P1}, we can get the results.
\qed
\end{proof}

Now we present the main result of the paper as follows.

\begin{theorem}\label{3T1}
Let $A \in B(X,Y)$ and let $T,\,T^\prime\subset X$, $S,\,S^\prime\subset Y$ be closed subspaces such that
$A_{T,S}^{(2)}$ exists and
$\max\{\hat{\delta}(T,T^\prime),\hat{\delta}(S,S^\prime)\}< \dfrac{1}{(1+\|A\|\|A_{T,S}^{(2)}\|)^2}$.
Then $A_{T^\prime,S^\prime}^{(2)}$ exists and
\begin{enumerate}
 \item[$(1)$] $A_{T',S'}^{(2)}=P_{T'}\big[I_X+P_{T'}(I+A_{T,S}^{(2)}P_{S^\perp}A(P_{T'}
 -P_{T}))^{-1}A_{T,S}^{(2)}(P_{S^\perp}P_{{S'}^\perp}-P_{S^\perp})AP_{T'}\big]^{-1}\\
\hspace*{1.3cm}\times P_{T'}(I_X+A_{T,S}^{(2)}P_{S^\perp}A(P_{T'}-P_{T}))^{-1}A_{T,S}^{(2)}
P_{S^\perp}P_{{S'}^\perp}.$
 \item[$(2)$] $\|A_{T^\prime,S^\prime}^{(2)}\|\leq \dfrac{\|A_{T,S}^{(2)}\|}
 {1-\|A_{T,S}^{(2)}\|\|A\|(\hat\delta(T,T^\prime)+\hat\delta(S,S'))}.$
\item[$(3)$] $\|A_{T^\prime,S^\prime}^{(2)}-A_{T,S}^{(2)}\|\leq \dfrac{1+\sqrt{5}}{2}
\dfrac{\|A_{T,S}^{(2)}\|^2\|A\|(\hat{\delta}(T,T^\prime)+\hat{\delta}(S,S^\prime))}
      {1-\|A_{T,S}^{(2)}\|\|A\|(\hat{\delta}(T,T^\prime)+\hat{\delta}(S,S^\prime))}.$
\end{enumerate}
\end{theorem}
\begin{proof} If $\hat{\delta}(T,T^\prime)< \dfrac{1}{(1+\|A\|\|A_{T,S}^{(2)}\|)^2},$ then by Proposition \ref{3P1},
$A_{T',S}^{(2)}$ exists and
\begin{align}
\label{3eqb}A_{T^\prime,S}^{(2)}&=P_{T^\prime}(I+A_{T,S}^{(2)}P_{S^\perp}A(P_{T^\prime}-P_{T}))^{-1}A_{T,S}^{(2)}P_{S^\perp}\\
\label{3eqc}\|A_{T^\prime,S}^{(2)}\|&\leq \frac{\|A_{T,S}^{(2)}\|}{1-\|A_{T,S}^{(2)}\|\|A\|\hat{\delta}(T,T^\prime)}
<\|A_{T,S}^{(2)}\|(1+\|A\|\|A_{T,S}^{(2)}\|)
\end{align}
for $\hat{\delta}(T,T^\prime)<\dfrac{1}{(1+\|A\|\|A_{T,S}^{(2)}\|)^2}\le\dfrac{1}{1+\|A\|\|A_{T,S}^{(2)}\|}$.

Noting that $\|A\|\|A_{T,S}^{(2)}\|\geq\|AA_{T,S}^{(2)}\|\geq 1$ and
$$
(1+\|A\|\|A_{T,S}^{(2)}\|)^2 \geq 2+\|A\|\|A_{T,S}^{(2)}\|(1+\|A\|\|A_{T,S}^{(2)}\|)>2+\|A\|\|A_{T',S}^{(2)}\|
$$
by (\ref{3eqc}), we have
$$
\hat{\delta}(S,S^\prime)<\dfrac{1}{(1+\|A\|\|A_{T,S}^{(2)}\|)^2}< \frac{1}{2+\|A\|\|A_{T^\prime,S}^{(2)}\|}.
$$
Hence $A_{T^\prime,S^\prime}^{(2)}$ exists with $\|A_{T^\prime,S^\prime}^{(2)}\|\leq
\dfrac{\|A_{T^\prime,S}^{(2)}\|}{1-\|A_{T^\prime,S}^{(2)}\|\|A\|\hat{\delta}(S,S^\prime)}$ and
$$
A_{T^\prime,S^\prime}^{(2)}=P_{T^\prime}(I_X+A_{T^\prime,S}^{(2)}(P_{(S^\prime)^\perp}-P_{S^\perp})
AP_{T^\prime})^{-1}A_{T^\prime,S}^{(2)}P_{(S^\prime)^\perp}
$$
by Proposition \ref{3P2}. Thus we have
\begin{align*}
A_{T^\prime,S^\prime}^{(2)}&=P_{T'}
\big[I_X+P_{T'}(I+A_{T,S}^{(2)}P_{S^\perp}A(P_{T'}-P_{T}))^{-1}A_{T,S}^{(2)}(P_{S^\perp}P_{{S'}^\perp}-P_{S^\perp})
AP_{T^\prime}\big]^{-1}\\
&\ \,\times P_{T'}(I+A_{T,S}^{(2)}P_{S^\perp}A(P_{T'}-P_{T}))^{-1}A_{T,S}^{(2)}P_{S^\perp}P_{{S'}^\perp}
\end{align*}
by (\ref{3eqb}) and
\begin{align*}
\|A_{T^\prime,S^\prime}^{(2)}\|&\leq \frac{1}{1-\frac{\|A_{T,S}^{(2)}\|}{1-\|A_{T,S}^{(2)}\|\|A\|\hat{\delta}(T,T^\prime)}\|A\|\hat{\delta}(S,S^\prime)}
\times \frac{\|A_{T,S}^{(2)}\|}{1-\|A_{T,S}^{(2)}\|\|A\|\hat{\delta}(T,T^\prime)}\\
&=\frac{\|A_{T,S}^{(2)}\|}{1-\|A_{T,S}^{(2)}\|\|A\|(\hat{\delta}(T,T^\prime)+\hat{\delta}(S,S^\prime))}.
\end{align*}
Moreover,
\begin{align*}
\|A_{T^\prime,S^\prime}^{(2)}-A_{T,S}^{(2)}\|&=\|A_{T^\prime,S^\prime}^{(2)}- A_{T^\prime,S}^{(2)}+ A_{T^\prime,S}^{(2)}-A_{T,S}^{(2)}\| \\
&\leq \|A_{T^\prime,S^\prime}^{(2)}- A_{T^\prime,S}^{(2)}\|+\| A_{T^\prime,S}^{(2)}-A_{T,S}^{(2)}\| \\
&\leq \frac{1+\sqrt{5}}{2}\|A_{T^\prime,S}^{(2)}\|\|A\|(\|A_{T^\prime,S^\prime}^{(2)}\|\hat{\delta}(S,S^\prime)+
\|A_{T,S}^{(2)}\|\hat{\delta}(T,T^\prime))\\
&\leq \frac{1+\sqrt{5}}{2}\frac{\|A_{T,S}^{(2)}\|\|A\|}{1-\|A_{T,S}^{(2)}\|\|A\|\hat{\delta}(T,T^\prime)}
(\|A_{T^\prime,S^\prime}^{(2)}\|\hat{\delta}(S,S^\prime)+\|A_{T,S}^{(2)}\|\hat{\delta}(T,T^\prime))\\
&\leq \frac{1+\sqrt{5}}{2}\frac{\|A_{T,S}^{(2)}\|\|A\|}{1-\|A_{T,S}^{(2)}\|\|A\|\hat{\delta}(T,T^\prime)}\\
&\ \,\times\bigg(\|A_{T,S}^{(2)}\|\hat{\delta}(T,T^\prime)+\frac{\|A_{T,S}^{(2)}\|}{1-\|A_{T,S}^{(2)}\|\|A\|(\hat{\delta}(T,T^\prime)+\hat{\delta}(S,S^\prime))}
\hat{\delta}(S,S^\prime)\bigg)\\
&=\frac{1+\sqrt{5}}{2}\frac{\|A_{T,S}^{(2)}\|^2\|A\|(\hat{\delta}(T,T^\prime)+\hat{\delta}(S,S^\prime))}
{1-\|A_{T,S}^{(2)}\|\|A\|(\hat{\delta}(T,T^\prime)+\hat{\delta}(S,S^\prime))}.
\end{align*}
\qed
\end{proof}

\begin{lemma}\label{3L2}
Let $A,\,\bar{A}=A+E\in B(X,Y)$ and let $T\subset X$, $S\subset Y$ be closed subspaces such that $A_{T,S}^{(2)}$ exists.
Suppose that $\|A_{T,S}^{(2)}\|\|E\|<1$. Then
$$
\bar{A}_{T,S}^{(2)}=(I_X+A_{T,S}^{(2)}E)^{-1}A_{T,S}^{(2)}=A_{T,S}^{(2)}(I_Y+EA_{T,S}^{(2)})^{-1}.$$
and
$$
\|\bar{A}_{T,S}^{(2)}\|\leq\frac{\|A_{T,S}^{(2)}\|}{1-\|A_{T,S}^{(2)}\|\|E\|},\quad
\|\bar{A}_{T,S}^{(2)}-A_{T,S}^{(2)}\|\leq \frac{\|A_{T,S}^{(2)}\|^2\|E\|}{1-\|A_{T,S}^{(2)}\|\|E\|}.
$$
\end{lemma}
\begin{proof} If $\|A_{T,S}^{(2)}\|\|E\|<1$, then $I_X+A_{T,S}^{(2)}E$ and $I_Y+EA_{T,S}^{(2)}$ are invertible.

Since $(I_X+A_{T,S}^{(2)}E)A_{T,S}^{(2)}=A_{T,S}^{(2)}(I_Y+EA_{T,S}^{(2)})$, it follows that
\begin{equation}\label{3eqd}
(I_X+A_{T,S}^{(2)}E)^{-1}A_{T,S}^{(2)}=A_{T,S}^{(2)}(I_Y+EA_{T,S}^{(2)})^{-1}.
\end{equation}
Put $B=(I_X+A_{T,S}^{(2)}E)^{-1}A_{T,S}^{(2)}$. From (\ref{3eqd}), we get that
$$
R(B)=R(A_{T,S}^{(2)})=T,\quad N(B)=N(A_{T,S}^{(2)})=S,\quad B(A+E)B=B.
$$
Therefore, $\bar{A}_{T,S}^{(2)}=(I_X+A_{T,S}^{(2)}E)^{-1}A_{T,S}^{(2)}$ and
$\|\bar{A}_{T,S}^{(2)}\|\leq\dfrac{\|A_{T,S}^{(2)}\|}{1-\|A_{T,S}^{(2)}\|\|E\|}.$

Since
$$
\bar{A}_{T,S}^{(2)}-A_{T,S}^{(2)}=(I_X+A_{T,S}^{(2)}E)^{-1}A_{T,S}^{(2)}-A_{T,S}^{(2)}
=-(I_X+A_{T,S}^{(2)}E)^{-1}A_{T,S}^{(2)}EA_{T,S}^{(2)},
$$
we have
$$\|\bar{A}_{T,S}^{(2)}-A_{T,S}^{(2)}\|\leq \frac{\|A_{T,S}^{(2)}\|^2\|E\|}{1-\|A_{T,S}^{(2)}\|\|E\|}.$$
\qed
\end{proof}

As an end of this section, we give the perturbation analysis for $A_{T,S}^{(2)}$ when $T$, $S$ and $A$ all have small
perturbation.
\begin{theorem}\label{3T2}
Let $A,\,\bar{A}=A+E\in B(X,Y)$ and let $T,\,T'\subset X$, $S,\,S'\subset Y$ be closed subspaces such that $A_{T,S}^{(2)}$ exists and
$$
\max\{\hat{\delta}(T,T^\prime),\hat{\delta}(S,S^\prime)\}< \frac{1}{(1+\|A\|\|A_{T,S}^{(2)}\|)^2}.
$$
If $\|A_{T,S}^{(2)}\|\|E\|<\dfrac{1}{1+\|A_{T,S}^{(2)}\|\|A\|}$, then
\begin{align*}
&\begin{aligned}
          (1)\ \bar{A}_{T^\prime,S^\prime}^{(2)}&
          =\{I_X+P_{T^\prime}[I_X+P_{T^\prime}(I+A_{T,S}^{(2)}P_{S^\perp}A(P_{T^\prime}-P_{T}))^{-1}A_{T,S}^{(2)}\\
          &\ \,\times(P_{S^\perp}P_{{S^\prime}^\perp}-P_{S^\perp})AP_{T^\prime}]^{-1}
         P_{T^\prime}(I_X+A_{T,S}^{(2)}P_{S^\perp}A(P_{T^\prime}-P_{T}))^{-1}A_{T,S}^{(2)}\\
         &\ \,\times P_{S^\perp}P_{{S^\prime}^\perp}E\}^{-1}P_{T^\prime}\{I_X+P_{T^\prime}(I+A_{T,S}^{(2)}P_{S^\perp}A
         (P_{T^\prime}-P_{T}))^{-1}\\
         &\ \,\times A_{T,S}^{(2)}(P_{S^\perp}P_{{S^\prime}^\perp}-P_{S^\perp})AP_{T^\prime}\}^{-1}\\
       &\ \,\times P_{T^\prime}(I_X+A_{T,S}^{(2)}P_{S^\perp}A(P_{T^\prime}-P_{T}))^{-1}A_{T,S}^{(2)}P_{S^\perp}P_{{S^\prime}^\perp},
\end{aligned} \\
&(2)\ \|\bar{A}_{T^\prime,S^\prime}^{(2)}\|\leq \frac{\|A_{T,S}^{(2)}\|}
     {1-\|A_{T,S}^{(2)}\|\big[\|E\|+\|A\|(\hat{\delta}(T,T^\prime)+\hat{\delta}(S,S^\prime))\big]}, \\
&(3)\ \|\bar{A}_{T^\prime,S^\prime}^{(2)}-A_{T,S}^{(2)}\|\leq \frac{\|A_{T,S}^{(2)}\|^2\big[\|E\|+\frac{1+\sqrt{5}}{2}\|A\|
(\hat{\delta}(T,T^\prime)+\hat{\delta}(S,S^\prime))\big]}
{1-\|A_{T,S}^{(2)}\|\big[\|E\|+\|A\|(\hat\delta(T,T^\prime)+\hat{\delta}(S,S^\prime))\big]}.
\end{align*}
\end{theorem}
\begin{proof} $A_{T^\prime,S^\prime}^{(2)}$ exists with
 $\|A_{T',S'}^{(2)}\|\leq \dfrac{\|A_{T,S}^{(2)}\|} {1-\|A_{T,S}^{(2)}\|\|A\|(\hat\delta(T,T')+\hat\delta(S,S'))}$
by Theorem \ref{3T1} when
$\max\{\hat{\delta}(T,T^\prime),\hat{\delta}(S,S^\prime)\}< \dfrac{1}{(1+\|A\|\|A_{T,S}^{(2)}\|)^2}$. Thus
$$
\|A_{T',S'}^{(2)}\|\|E\|\leq \dfrac{\|E\|\|A_{T,S}^{(2)}\|} {1-\|A_{T,S}^{(2)}\|\|A\|(\hat\delta(T,T')+\hat\delta(S,S'))}
<\frac{1+\|A_{T,S}^{(2)}\|\|A\|}{1+(\|A_{T,S}^{(2)}\|\|A\|)^2}\leq 1,
$$
that is, $\|A_{T',S'}^{(2)}\|\|E\|<1$ by above inequalities for $\|A_{T,S}^{(2)}\|\|A\|\geq\|A_{T,S}^{(2)}A\|\geq 1$.
Consequently, $\bar{A}_{T',S'}^{(2)}=(I_X+A_{T',S'}^{(2)}E)^{-1}A_{T',S'}^{(2)}$ by Lemma \ref{3L2}.
Simple computation shows that
\begin{align*}
\|\bar{A}_{T',S'}^{(2)}\|&\leq \frac{\|A_{T,S}^{(2)}\|}{1-\|A_{T,S}^{(2)}\|\{\|E\|+\|A\|(\hat{\delta}(T,T^\prime)+
\hat{\delta}(S,S^\prime))\}},\\
\bar{A}_{T^\prime,S^\prime}^{(2)}&=
\{I_X+P_{T^\prime}[I_X+P_{T^\prime}(I+A_{T,S}^{(2)}P_{S^\perp}A(P_{T^\prime}-P_{T}))^{-1}A_{T,S}^{(2)}(P_{S^\perp}
P_{{S^\prime}^\perp}-P_{S^\perp})\\
 &\ \,\times AP_{T^\prime}]^{-1}P_{T^\prime}(I_X+A_{T,S}^{(2)}P_{S^\perp}A(P_{T^\prime}-P_{T}))^{-1}A_{T,S}^{(2)}P_{S^\perp}
 P_{(S^\prime)^\perp}E\}^{-1}\\
&\ \,\times P_{T^\prime}\{I_X+P_{T^\prime}(I_X+A_{T,S}^{(2)}P_{S^\perp}A(P_{T^\prime}-P_{T}))^{-1}A_{T,S}^{(2)}(P_{S^\perp}
 P_{{S^\prime}^\perp}-P_{S^\perp})\\
&\ \,\times AP_{T^\prime}\}^{-1}
P_{T^\prime}(I_X+A_{T,S}^{(2)}P_{S^\perp}A(P_{T^\prime}-P_{T}))^{-1}A_{T,S}^{(2)}P_{S^\perp}P_{{S^\prime}^\perp}.
\end{align*}

Noting that
\begin{align*}
\bar{A}_{T^\prime,S^\prime}^{(2)}-A_{T,S}^{(2)}&=(I_X+A_{T^\prime,S^\prime}^{(2)}E)^{-1}A_{T^\prime,S^\prime}^{(2)}-A_{T,S}^{(2)}\\
&=(I_X+A_{T^\prime,S^\prime}^{(2)}E)^{-1}(A_{T^\prime,S^\prime}^{(2)}-(I_X+A_{T^\prime,S^\prime}^{(2)}E)A_{T,S}^{(2)})\\
&=(I_X+A_{T^\prime,S^\prime}^{(2)}E)^{-1}(A_{T^\prime,S^\prime}^{(2)}-A_{T,S}^{(2)}-A_{T^\prime,S^\prime}^{(2)}EA_{T,S}^{(2)}),
\end{align*}
we have
\begin{align*}
\|\bar{A}_{T^\prime,S^\prime}^{(2)}-A_{T,S}^{(2)}\|&\leq \|(I_X+A_{T^\prime,S^\prime}^{(2)}E)^{-1}\|(\|A_{T^\prime,S^\prime}^{(2)}-A_{T,S}^{(2)}\|+\|A_{T^\prime,S^\prime}^{(2)}EA_{T,S}^{(2)}\|)\\
&\leq \frac{1}{1-\|A_{T^\prime,S^\prime}^{(2)}\|\|E\|}(\|A_{T^\prime,S^\prime}^{(2)}-A_{T,S}^{(2)}\|+\|A_{T^\prime,S^\prime}^{(2)}\|\|E\|\|A_{T,S}^{(2)}\|)\\
&\leq\frac{\|A_{T,S}^{(2)}\|^2\Big[\|E\|+\frac{1+\sqrt{5}}{2}\|A\|(\hat{\delta}(T,T^\prime)+\hat{\delta}(S,S^\prime))\Big]}
{1-\|A_{T,S}^{(2)}\|\big[\|E\|+\|A\|(\hat{\delta}(T,T^\prime)+\hat{\delta}(S,S^\prime))\big]}.
\end{align*}
\qed
\end{proof}

\vskip0.2cm \noindent{\bf{Acknowledgement.}} The authors thank to
the referee for his (or her) helpful comments and suggestions.


\begin{thebibliography}{99}
\bibitem{BG}A. Ben-Israel and T.N.E. Greville, Generalized inverse: Theory and Applications,
Springer--Verlag, New York, 2003.

\bibitem{CX} G. Chen and Y. Xue, Perturbation analysis for the operator equation $Tx=b$ in Banach spaces.
 J. Math. Anal. Appl., 212 (1997), 107-125.

\bibitem{CWX} G. Chen, M. Wei and Y. Xue, Perturbation analysis of the least square solution in Hilbert spaces,
Linear Alebra Appl., 244 (1996), 69-80.

\bibitem{YC}Y. Chen, Iterative methods for computing the generalized inverse $A^{(2)}_{T,S}$ of a matrix $A$,
Appl. Math. Comput. 75 (1996), 207-222.

\bibitem{DS}D. Djordjevi\'{c} and Stanimirovi\'{c}, Splitting of operators and generalized inverses,
Publ. Math. Debrecen, 59 (2001), 147--159.

\bibitem{DSW}D. Djordjevi\'{c}, Stanimirovi\'{c} and Y. Wei,
The representation and approximations of outer generalized inverses, Acta Math. Hungar., 104(1--2) (2004), 1--26.

\bibitem{DX}F. Du and Y. Xue, Perturbation analysis of $A^{(2)}_{T,S}$ on banach spaces, Electronic J. Linear Algebra,
23 (2012), 586-598.

\bibitem{DX1}F. Du and Y. Xue, Note on stable perturbation of bounded linear operators on Hilbert spaces,
Funct. Anal. Appr. Comput., 3(2) (2011), 47--56.

\bibitem{TK}T. Kato, Perturbation Theory for Linear Operators, Springer--Verlag, New York, 1984.

\bibitem{VM} V. M\"{u}ller, Spectral Theory of Linear Operators and spectral systems in Banach algebras.
                         Birkh\"{a}user Verlag AG, 2nd Edition, (2007).

\bibitem{YW}Y. Wei, A characterization and representation for the generalized inverse $A^{(2)}_{T,S}$ and its
applications, Linear Algebra Appl. 280 (1998), 87-96.

\bibitem{WW1}Y. Wei and H. Wu, On the perturbation and subproper splittings for the generalized inverse $A^{(2)}_{T,S}$
of rectangular matrix A, J. Comput. Appl. Math., 137 (2001), 317--329.

\bibitem{WW2}Y. Wei and H. Wu, (T--S) Splitting methods for computing the generalized inverse $A^{(2)}_{T,S}$ and rectangular systems,
                International Journal of Computer Mathematics, 77(2001), 401--424.

\bibitem{XUE}Y. Xue, Stable perturbations of operators and related topics, World Scientific, 2012.

\bibitem{XCC} Y. Xue and G. Chen, Some equivalent conditions of stable perturbation of operators in Hilbert spaces,
                         Applied Math. comput., 147 (2004), 765--772.

\bibitem{ZW1} N. Zhang and Y. Wei, Perturbation bounds for the generalized inverse $A^{(2)}_{T,S}$ with application to
constrained linear system, Appl. Math. Comput., 142 (2003), 63--78.

\bibitem{ZW2} N. Zhang and Y. Wei, A note on the perturbation of an outer inverse, Calcolo. 45(4) (2008), 263--273.

\end{thebibliography}
\end{document}